\documentclass[12pt]{amsart}
\usepackage{graphicx}
\usepackage{amssymb}
\usepackage{stmaryrd}
\usepackage{enumitem}
\usepackage{lmodern}
\usepackage{geometry}
\usepackage{tikz}
\usepackage{float}
\usepackage{url}
\usepackage{mathtools}
\usepackage{amsmath,amsfonts,amsthm}
\usepackage{mathrsfs} 
\newcommand{\irr}{\operatorname{irr}}
\newcommand{\case}[1]{\paragraph*{Case #1:}}
\newtheorem{hypothesize}{Hypothesize}
\usepackage{mathtools}

\usepackage{amsmath}
\usepackage{systeme}
\newtheorem{theorem}{Theorem}[section]
\newtheorem{lemma}[theorem]{Lemma}
\newtheorem{proposition}{Proposition}[section]

\newtheorem{definition}{Definition}

\author{Jasem Hamoud}
\address{\textbf{Jasem Hamoud} 
Department of Discrete Mathematics, Moscow Institute of Physics and Technology
}
\email{khamud@phystech.edu}
\thanks{}

\author{Alexei Belov-Kanel}
\address{\textbf{Alexei Belov-Kanel} 
Department of Mathematics, Bar-Ilan University, Ramat Gan, 5290002, Israel\\
Department of Discrete Mathematics, Moscow Institute of Physics and Technology, Dolgoprudnyi, Institutskiy Pereulok,
141700 Moscow, Russia
}
\email{kanelster@gmail.com}
\thanks{}

\author{Duaa Abdullah}
\address{\textbf{Duaa Abdullah:} Department of Discrete Mathematics, Moscow Institute of Physics and Technology }
\email{abdulla.d@phystech.edu}
\thanks{}

\title{Bounds on Trees with Topological Indices Among Degree Sequence}

\date{}
\begin{document}

\begin{abstract}
In this paper, we investigate The relationship between the Albertson index and the first Zagreb index for trees.  For a tree $T=(V,E)$ with $n=|V|$ vertices and $m=|E|$ edges, we provide several bounds and exact formulas for these two topological indices, and we show that the Albertson index $\irr(T)$ and the first Zagreb index $M_1(T)$ satisfy the association \[ \irr(T)=d_1^2+d_n^2+(n-2)\left(\frac{\Delta + \delta}{2}\right)^2+\sum_{i=2}^{n-1} d_i+d_n - d_1-2n-2.\]
Our goal of this paper is provide a topological indices, Albertson index, Sigma index among a degree sequence $\mathscr{D}=(d_1,\dots,d_n)$  where it is non-increasing and non-decreasing of tree $T$.
\end{abstract}

\maketitle

\noindent\rule{12.7cm}{1.0pt}


\medskip

\noindent
%{\bf MSC 2020:} 05C05, 05C12, 05C20, 05C25, 05C35, 05C76, 68R10.

\medskip


\noindent\rule{12.7cm}{1.0pt}

\section{Introduction}\label{sec1}
Throughout this paper. Let $G=(V,E)$ be a simple graph of order $n$, with vertices set $V=\{v_1,v_2,\dots,v_n\}$ and edges set $E=\{e_1,e_2,\dots,e_m\}$.  Topological indices are one of the many graph invariants investigated in this discipline, and they play an important role in measuring graph structure.  These indices act as numerical descriptors, capturing key properties of molecular structures, networks, and other graph-based systems.

In 1995, Molloy and Reed in~\cite{Molloy} defined a degree sequence as $\mathscr{D}=(d_1,d_2,\dots,d_n)$. Zhang et al. in~\cite{Zhang} define the classes of trees $\mathcal{T}$ among a non-increasing degree sequence $\mathscr{D}$, denote by $\mathcal{T}_{\mathscr{D}}$ when $d_1\geqslant d_2\geqslant\dots \geqslant d_n$, in~\cite{Blanc} define $\mathcal{T}_{\mathscr{D}}$ as $\sum_{i=1}^{n}d_i=n-1$. The maximum degree $\Delta$ of any tree with a degree sequence $\mathscr{D}$ defined by Broutin et al. in~\cite{Broutin} as $\Delta_{\mathscr{D}}=\max \{i, \mathcal{T}_{\mathscr{D}}>0\}$. Marshall, A.W., et al in~\cite{Marshall79Olkin} introduce two non-increasing vectors in $\mathbb{R}^n$ depending on a degree sequence.\par 
%---------------------------------------------------
Furthermore, let $\deg v=\deg_G(v)$ be the degree of a vertex $v\in V(G)$, the amount of divergence at $u$ equals $\deg u$, which for any tree has a center chord of a maximum of one or two points. For a graph $G$ with $n$ vertices and $p$ pendents. For each pendent $G_i$ (with $n_i$ vertices), the maximum edges within it are $\frac{1}{2}n_i(n_i - 1)$ (see also \cite{Harary}). For any edge connecting the vertices $u, v \in V$  and  $\{u,v\} \in E$.  Kunegis,J. in~\cite{KunegisJ} show that for $u,v \in V$ are connected, then denoted by $u \sim v$. The degree of $u$ is a number of neighbour of a node $u$ and is called by $\deg (u)=\{v \in V \mid u \sim v\}$. In~\cite{Gallaipp,Brandt}
for $\mathscr{D}=(d_1,\dots,d_n)$ be a degree sequence  where $d_1\geqslant d_2\geqslant \dots \geqslant d_n$, then the inequality satisfying: 
\begin{equation}~\label{eq1}
    \sum_{i=1}^{n}d_i\leq n(n-1)+\sum_{i=n+1}^{k} \min(n,d_i).
\end{equation}

 In 1997, Albertson in~\cite{ALBERTSON} mentions the imbalance of an edge $uv$ by $imb(uv)$ where $imb(uv)=\lvert d_u-d_v\rvert$.We considered that a graph $G$ is regular if all of its vertices have the exact same degree as we show in~(\ref{eq1}), then the irregularity measure defined in~\cite{ALBERTSON, GUTMAN2, Brandt} as:
\[
\operatorname{irr}(G)=\sum_{uv\in E(G)}\lvert d_u(G)-d_v(G) \rvert.
\]
\begin{lemma}\cite{Nasiri}~\label{lem1}
	The star $S_n$  is the only tree of order n that has a great deal of irregularity, satisfying:
	\[\operatorname{irr}\left( {{S_n}} \right) = \left( {n - 2} \right)\left( {n - 1} \right).\]
\end{lemma} 
%---------------------------------------------------
 Dorjsembe et al. In~\cite{Dorjsembe} given the relationship between irregularity of Albertson index and minimum, maximum degrees $\delta,\Delta$ of graph $G$ (respectfully), where contribute vital roles in determining connection, shading, component incorporation, and realisability.  In~\cite{Hosam,Brandt} defined total irregularity as $\operatorname{irr}_T(G)=\sum_{\{u,v\} \subseteq V(G)}\lvert d_u(G)-d_v(G) \rvert$ and shown upper bounds on the irregularity of graphs. When $G$ is a tree, Ghalavand, A. et al. In~\cite{Ghalavand} proved  $\operatorname{irr}_T(G) \leq \frac{n^2}{4} \operatorname{irr}(G)$ when the bound is sharp for infinitely many graphs.  Andrade, E., et al. in~\cite{AndradeRobbianoLenes} introduce a topological index of graph $G$ with $m \geq 1$ edges, it known ``The Zagreb index'' of $G$, denote by $Z_{g}(G)$, then $Z_{g}(G)=\sum_{i \in V(G)} d^{2}(i)$, for $i=1,2, \ldots, k$, where $k=\min \{n, m\}$.The first and the second Zagreb index, $M_1(G)$ and $M_2(G)$ are defined in~\cite{GutmanTrinajstic, WILCOX,ALBERTSON,Liu} as: 
\[
M_1(G)=\sum_{i=1}^{n}d_i^2, \quad \text{and} \quad M_2(G)=\sum_{uv\in E(G)} d_u(G)d_v(G).
\] 
Ghalavand, A., et al. in~\cite{Ghalavand2023Tavakoli} introduced the relationship of the Albertson index with the Zagreb index. The Zagreb index has been extensively explored in chemical graph theory, with applications in quantitative structure-property relationship (QSPR) and quantitative structure-activity relationship (QSAR) research. Therefore, in~\cite{Kexiangsmallest, Dorjsembe,Sigarreta2021Rada,Gutman2024Gutman,Dimitrov2023Stevanovi, Furtula} provides many bounds for  any tree of order $n$ and maximum degree $\Delta$ depend on definition of the first Zagreb index as:
\begin{equation}~\label{eq2}
M_1(G)\leq \max \left\{(n-1)\left(\Delta+\frac{n-1}{\Delta}\right), \frac{(n-1)(n+3)}{2}\right\}.
\end{equation}
Another important topological descriptor is the first Zagreb index, denoted by $M_1(G)$, which was introduced by Gutman and Trinajstić \cite{GutmanTrinajstic}. For a graph $G$, the first Zagreb index is defined as:

\begin{equation}~\label{eq3}
M_1(G) = \sum_{v \in V} d_v^2 = \sum_{uv \in E} (d_u + d_v)
\end{equation}
%---------------------------------------------------
 The recently introduced $\sigma(G)$ irregularity index is a simple diversification of the previously established Albertson irregularity index, in~\cite{GutmanTogan,DimitrovAbdo} defined as: 
\[
\sigma(G)=\sum_{uv\in E(G)}\left( d_u(G)-d_v(G) \right)^2.
\]
Trees, or connected acyclic graphs, are a fundamental type of graph structure with applications ranging from computer science methods to chemical structures.  The study of topological indices on trees has received a lot of attention because of its simple but versatile structure.  For a tree $T$ with $n$ vertices, certain features are readily apparent:  $T$ has precisely $n-1$ edges, is linked, and does not include any cycles.\par 
In this paper, we investigate the relationship between the Albertson index and the first Zagreb index for trees. Specifically, we establish several bounds and inequalities that relate these two indices, taking into account parameters such as the maximum degree $\Delta$ and minimum degree $\delta$ of the tree. Our results extend previous work by providing tighter bounds and more general relationships.\par 
%---------------------------------------------------
The goal of this paper is provide a topological indices, Albertson index, Sigma index among a degree sequence $\mathscr{D}=(d_1,\dots,d_n)$  where it is non-increasing and non-decreasing of tree $T$. This paper is organized as follows. In Section~\ref{sec1}, we observe the important concepts to our work including literature view of most related papers, in Section~\ref{sec2} we have provided a preface through some of the important theories we have utilised in understanding the work based on literature view of Section~\ref{sec1}, in Section~\ref{sec3} we provide the main result of our paper, where divided into subsections, in subsection~\ref{subsec1} we introduced the first Zagreb index with bounds on trees according to Albertson index. In Section~\ref{sec4}  we provide bounds on trees with Sigma Index according to maximum degree $\Delta$ and minmum degree $\delta$ of tree $T$. 

%==================================
\section{Preliminaries}\label{sec2}
%==================================

In this section, we presented several fundamental concepts that will be deployed for the major results for $\mathscr{D}=(d_1,\dots,d_n)$ a degree sequence in Theorem~\ref{hy.1}. Also, we provide concept of Sigma index in Theorem~\ref{sigmathm}.
\begin{definition}[Star Tree]~\label{StardefTree}
Let $T$ be a tree of order $n$,if $T$ has a central vertices $v_0$ called the root linked with pendent vertices $(v_i)_{i\geq 1}$ where $\deg_T(v_i)=1$, denote $T$ by $S_n$ a star tree of order $n$. 
\end{definition}
\begin{definition}[Starlike Tree]~\label{StarlikedefTree}
Let $T$ be a tree of order $n$,if $T$ has a central vertices $n_0$ called the root where $\deg_T(n_0)\geqslant 3$, linked to $k\geqslant 3$ paths of lengths $n_1,n_2,\dots,n_k$ where $n_i\geq 1$, denote $T$ by $S_k$ a Starlike tree of order $n$ by $S(n_1,n_2,\dots,n_k)$. 
\end{definition}
We can demonstrate that in Figure~\ref{fig:StarlikeTree} for Star Tree and Starlike Tree as we show that.
\begin{figure}[H]
    \centering
   \begin{tikzpicture}[scale=0.8]
  % Star graph S_5
  \begin{scope}[xshift=-5cm]
 % Central vertex
    \node[circle, draw, fill=black, inner sep=2pt, label=below:{$v_0$}] (v0) at (0,0) {};
    % Leaf vertices
    \node[circle, draw, fill=black, inner sep=2pt, label=above:{$v_1$}] (v1) at (0,2) {};
    \node[circle, draw, fill=black, inner sep=2pt, label=above right:{$v_2$}] (v2) at (1.7,1) {};
    \node[circle, draw, fill=black, inner sep=2pt, label=below right:{$v_3$}] (v3) at (1.7,-1) {};
    \node[circle, draw, fill=black, inner sep=2pt, label=below left:{$v_4$}] (v4) at (-1.7,-1) {};
    \node[circle, draw, fill=black, inner sep=2pt, label=above left:{$v_5$}] (v5) at (-1.7,1) {};
    % Edges
    \draw (v0) -- (v1);
    \draw (v0) -- (v2);
    \draw (v0) -- (v3);
    \draw (v0) -- (v4);
    \draw (v0) -- (v5);
        % Label for S_5
    \node at (0,-2.5) {$S_5$ (Star Graph)};
  \end{scope}
    % Starlike tree S(3,2,2,2,2,2)
  \begin{scope}[xshift=5cm]
    % Central vertex
    \node[circle, draw, fill=black, inner sep=2pt, label=below:{$u_0$}] (u0) at (0,0) {};
    \node[circle, draw, fill=black, inner sep=2pt] (u1) at (0,1) {};
    \node[circle, draw, fill=black, inner sep=2pt] (u2) at (0,2) {};
    \node[circle, draw, fill=black, inner sep=2pt] (u3) at (0,3) {};
    \node[circle, draw, fill=black, inner sep=2pt] (u4) at (1.5,0.5) {};
    \node[circle, draw, fill=black, inner sep=2pt] (u5) at (2.5,1) {};
    \node[circle, draw, fill=black, inner sep=2pt] (u6) at (1.5,-0.5) {};
    \node[circle, draw, fill=black, inner sep=2pt] (u7) at (2.5,-1) {};
    \node[circle, draw, fill=black, inner sep=2pt] (u8) at (0,-1) {};
    \node[circle, draw, fill=black, inner sep=2pt] (u9) at (0,-2) {};
    \node[circle, draw, fill=black, inner sep=2pt] (u10) at (-1.5,-0.5) {};
    \node[circle, draw, fill=black, inner sep=2pt] (u11) at (-2.5,-1) {};
    \node[circle, draw, fill=black, inner sep=2pt] (u12) at (-1.5,0.5) {};
    \node[circle, draw, fill=black, inner sep=2pt] (u13) at (-2.5,1) {};
    % Edges
    \draw (u0) -- (u1) -- (u2) -- (u3);
    \draw (u0) -- (u4) -- (u5);
    \draw (u0) -- (u6) -- (u7);
    \draw (u0) -- (u8) -- (u9);
    \draw (u0) -- (u10) -- (u11);
    \draw (u0) -- (u12) -- (u13);
    % Label for S(3,2,2,2,2,2)
    \node at (0,-3) {$S(3,2,2,2,2,2)$ (Starlike Tree)};
  \end{scope}
\end{tikzpicture}
    \caption{$S_5$ (Star Graph) and $S(3,2,2,2,2,2)$ (Starlike Tree)}
    \label{fig:StarlikeTree}
\end{figure}
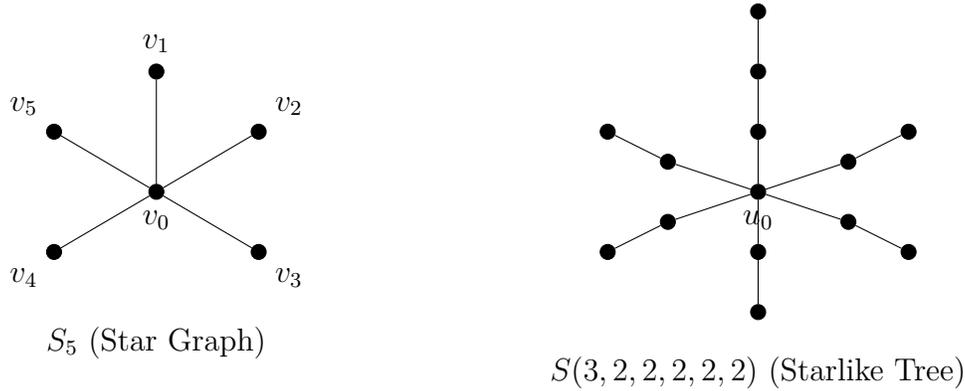
\begin{hypothesize}{\label{hy.2}}
Let $\mathscr{D}=(d_1,\dots,d_n)$ be a degree sequence with $n\geq 3$ and  let be order as: ${d_n} > d_{n-1}> \dots> d_2>  d_1$, the caterpillar tree with such order has the  minimum  value  of $\irr$ among all caterpillar trees with such degrees sequence of path vertices.
\end{hypothesize}
\begin{theorem}{\label{hy.1}}
Let $\mathscr{D}=(d_1,\dots,d_n)$ be a degree sequence with $n\geq 3$ and  let be order as: $d_n > d_1>\dots > d_2 > d_{n-1}$, the caterpillar tree with such order has the  maximum value  of $\irr$ among all caterpillar trees with such degrees sequence of path vertices.
 \end{theorem}
 \begin{theorem}[Sigma index ~\label{sigmathm}]\cite{GutmanTogan}
For every simple graph $G$, then Sigma index $\sigma(G)$ is an even integer, so that we have:
\[
\sigma(G)=F(G)-2 M_2(G)=\sum_{uv\in E(G)}\left[\left(d_u\right)^2+\left(d_v\right)^2\right]-2 \sum_{uv \in \in(G)} d_u d_v .
\]
where $M_2(G)=\sum_{uv\in E(G)} d_ud_v$.
\end{theorem}
\begin{theorem}~\label{pro.se.1}
Let $\mathcal{T}_1,\mathcal{T}_2$ be a class of trees of order $n$, let $\mathscr{D}_i=(x_1,x_2,\dots, x_{n-1})$ and $\mathscr{D}_j=(y_1,y_2,\dots, y_{n-1})$ be tow non increasing degree sequence, where $\sum_{i=0}^{n-1}d_i\leqslant \sum_{j=0}^{n-1}d_j$, then we have: 
\[
\irr(\mathcal{T}_{\mathscr{D}_i})\leqslant \irr(\mathcal{T}_{\mathscr{D}_j}).
\]
\end{theorem}
\begin{lemma}~\label{le.sigma2}
Let $T$ be a tree of order $n>0$ and a degree sequence $d=(d_1,d_2,d_3,d_4)$ where $d_4\geqslant d_3\geqslant d_2\geqslant d_1$, then Sigma index is: 
\[
\sigma(T)=\sum_{i=1}^{2} (d_i+1)(d_i-1)^2+\sum_{i=1}^{4}(d_i-d_{i+1})^2+\sum_{i=2}^{3}(d_i+2)(d_i-1)^2-2.
\]
\end{lemma}
\begin{proposition}~\label{classoftreessigma}
Let $\mathcal{T}_{n, \Delta}$ be a class of trees with $n$ vertices, let $T \in \mathcal{T}_{n, \Delta}$ be a tree, and let $v_0 \in V(T)$ be a vertex with maximum degree $\Delta$. For any support vertex $v_{\ell}$ in $T$, different from $v_0$, where $3<\deg(v_{\ell})<10$, then there exists another tree $T^{\prime} \in \mathcal{T}_{n, \Delta}$  that $\sigma(T^{\prime}) < \sigma(T)$.
\end{proposition}
\begin{theorem}~\label{mainalb2}
    Let  $T$ be a tree of order $n$, a degree sequence is $\mathscr{D}=(d_1,\dots,d_n)$ where $d_n\geqslant \dots \geqslant d_1$, then Albertson index of tree $T$ is: 
    \[
    \irr(T)=d_1^2+d_n^2+\sum_{i=2}^{n-1} d_i^2+\sum_{i=2}^{n-1} d_i+d_n - d_1-2n-2.
    \]
\end{theorem}
\begin{theorem}~\label{thm.sigma}
Let $T$ be a tree of order $n$,let $\mathscr{D}=(d_1,\dots,d_n)$ be a degree sequence, where $d_n\geqslant \dots \geqslant d_1$, then Sigma index among tree given as:
\[
\sigma(T)=\sum_{i\in\{1,n\}}(d_i+1)(d_i-1)^2+\sum_{i=2}^{n-1}(d_i+2)(d_i-1)^2+\sum_{i=2}^{n-1}(d_i-d_{i+1})^2+2n-2.
\]
\end{theorem}
 %===================
\section{Main Result}\label{sec3}
%====================
In this section, we provided concerning the Albertson index and in particular a degree sequence from motivation in Proposition~\ref{prozagn1},\ref{prozagn2} which it employed in Theorems~\ref{mainalb3}, \ref{mainalb4} and \ref{thmalbn1}.

%==============================
\subsection{Bounds on Trees with Albertson Index}~\label{subsec1}
%=============================
\begin{proposition}~\label{prozagn1}
Let $T=(V,E)$ be a tree with $n=|V|$ vertices and $m=|E|$ edges, let $M_1(T)$ be the first Zagreb index with maximum degree $\Delta$, then
\[
\irr(T)\geq M_1(T)+\Delta(\Delta-1).
\]
\end{proposition}
\begin{proposition}~\label{prozagn2}
Let $T=(V,E)$ be a tree where vertices set $V=\{v_1,\dots,v_i\}$ with $n=|V|$ vertices and edges set $E=\{e_1,\dots,e_j\}$ with $m=|E|$ edges, let $\mathscr{D}=(d_1,\dots,d_n)$ be a degree sequence, where $d_n\geqslant d_{n-1}\geqslant \dots \geqslant d_1$, let $M_1(T)$ be the first Zagreb index with maximum degree $\Delta$, then Albertson index satisfying 
\[\irr(T)\geqslant 2\sum_{v_i\in V(T)}\binom{d_i}{2}+\frac{4m}{n-1}-\Delta(\Delta-1).\]
\end{proposition}
\begin{proof}
Let $T=(V,E)$ be a tree where vertices set $V=\{v_1,\dots,v_i\}$ with $n=|V|$ vertices, let $\mathscr{D}=(d_1,\dots,d_n)$ be a degree sequence, where $d_n\geqslant d_{n-1}\geqslant \dots \geqslant d_1$, then we show in proof of Theorem~\ref{mainalb3} that $d_1^2+d_n^2\leq 2m/(n-1)$. Actually, the first Zagreb index\cite{Liu} known by
\begin{equation}~\label{eqqtrzagn1}
    M_1(T)=2\sum_{v_i\in V(T)}\binom{d_i}{2}+2m.
\end{equation}
From~(\ref{eqzagn3}) according to Proposition~\ref{prozagn1} we found $\irr(T)\geq M_1(T)+\Delta(\Delta-1)$. Thus, clearly
\begin{equation}~\label{eqqtrzagn2}
    \irr(T)\geq M_1(T).
\end{equation}
Therefore, when $d_1^2+d_n^2\leq 2m/(n-1)$ satisfying $(d_1+d_n)^2\leq 4m/(n-1)^2$, then we have $\irr(T)>4m/(n-1)^2$ holds to $\irr(T)>4m/(n-1)$, from~(\ref{eqqtrzagn1}),(\ref{eqqtrzagn2}) we have
\begin{equation}~\label{eqqtrzagn3}
\irr(T)\geqslant 2\sum_{v_i\in V(T)}\binom{d_i}{2}+\frac{4m}{n-1}.
\end{equation}
We need to identify the association for the maximum degree $\Delta$ given to prove the validity of the Albertson association by a degree sequence under the given condition $\Delta(\Delta-1)$. For this, we have: 
\case{1} If $\Delta=n-1$, from~(\ref{eqzagn1}) we have
\begin{equation}~\label{eqtrzangn4}
    m\left(\frac{2m}{\Delta}+\Delta-1\right)=(\Delta+1)(2m-(\Delta+2))\leq (\Delta+1)(\Delta+2).
\end{equation}
\case{2} If $\Delta=n$, then, this case by comparing with~(\ref{eqtrzangn4}) satisfying as
\begin{equation}~\label{eqtrzangn5}
    m\left(\frac{2m}{\Delta-1}+\Delta-2\right)=(\Delta)(2m-\Delta+1)\leq -\Delta(\Delta-1).
\end{equation}
\case{3} If $\Delta=n+i$ where $i\in \mathbb{N}$, for one case we consider $i=1$, then 
\begin{equation}~\label{eqtrzangn6}
    m\left(\frac{2m}{\Delta-2}+\Delta-3\right)=(\Delta-1)(2m-\Delta)\leq \Delta(\Delta-1).
\end{equation}
Thus, from~(\ref{eqtrzangn4}),(\ref{eqtrzangn5}) and (\ref{eqtrzangn6}) we have:
\[
2\sum_{v_i\in V(T)}\binom{d_i}{2}+\frac{4m}{n-1}\geqslant \Delta(\Delta-1).
\]
As desire.
\end{proof}

\begin{proposition}~\label{prozagrebn1}
Let $T$ be a tree with maximum degree $\Delta$ and minmum degree $\delta$, let $\mathscr{D}=(d_1,d_2,\dots,d_n)$ be a degree sequence, if $d_2=\dots=d_{n-1}=(\Delta+\delta)/2$, then
\[
\irr(T)+ M_1(T) \geqslant \sum_{i=2}^{n-1} d_i+\left(\frac{\Delta + \delta}{2}\right)^2.
\]
\end{proposition}
\begin{proof}
Assume the degree sequence is $\mathscr{D}_1=(\Delta, \lambda, \lambda, \dots, \lambda, \delta)$, where $\lambda=(\Delta+\delta)/2$  and $n-2$ vertices of degree $\lambda$. Let $\mathscr{D}_2=(d_1,d_2,\dots,d_n)$ be a degree sequence, where $d_2=\dots=d_{n-1}=(\Delta+\delta)/2$. In fact, there are $n-2$ terms, then we need to prove this inequality:
\begin{equation}~\label{proveineq}
    \irr(T)+ M_1(T) \geqslant \sum_{i=2}^{n-1} d_i+\left(\frac{\Delta + \delta}{2}\right)^2.
\end{equation}
Hence, the sum of degree sequence $\mathscr{D}_2$ is 
\begin{equation}~\label{satrn1}
\sum_{i=2}^{n-1} \frac{\Delta + \delta}{2} = \frac{(n-2)(\Delta + \delta)}{2}.
\end{equation}
According to~(\ref{satrn1}) noticed that $\Delta + \delta = 4 - 4/n$. Then, by considering the degree sum given by:
\begin{equation}~\label{satrn2}
    \Delta + (n-2)\lambda + \delta = 2(n-1), \quad \Delta + \delta = 2\lambda.
\end{equation}
Thus, $n\lambda = 2(n-1)$, then $\lambda=2-2/n$. According to Theorem~\ref{thmsign1}, from~(\ref{satrn1}),(\ref{satrn2}) we have
\begin{equation}~\label{satrn3}
    M_1(T)-(\Delta^2+\delta^2) =(n-2)\left(\frac{\Delta + \delta}{2}\right)^2.
\end{equation}
Hence, 
\begin{gather*}
  M_1(T)-(\Delta^2+\delta^2)=(n-2)\left(\frac{\Delta + \delta}{2}\right)^2\\
  \sum_{i=1}^{n}d_i^2-(\Delta^2+\delta^2)  =\sum_{i=2}^{n-1}d_i^2.
\end{gather*}
Then we have: 
\begin{equation}~\label{resalbn1}
  \sum_{i=1}^{n}d_i^2- \sum_{i=2}^{n-1}d_i^2= d_1^2+d_n^2.
\end{equation}
If the degree sequence $\mathscr{D}_2$ is non-increasing then from~(\ref{resalbn1}) we find that $d_1^2+d_n^2$ is a maximum, according to~(\ref{satrn3}) we have: 
\begin{equation}~\label{resalbn2}
M_1(T) \geqslant \left(\frac{\Delta + \delta}{2}\right)^2.
\end{equation}
Now, we need to prove $\irr(T) \geq \sum_{i=2}^{n-1} d_i$, thus, we find from~(\ref{satrn1}) that the sum is $(n-2)(\Delta+\delta)/2$, where Albertson index given by: 
\[
 \irr(T)=d_1^2+d_n^2+\sum_{i=2}^{n-1} d_i^2+\sum_{i=2}^{n-1} d_i+d_n - d_1-2n-2.
\]
Thus, 
\[
 \irr(T)=d_1^2+d_n^2+(n-2)\left(\frac{\Delta + \delta}{2}\right)^2+\sum_{i=2}^{n-1} d_i+d_n - d_1-2n-2.
\]
As desire.
\end{proof}

\begin{theorem}~\label{mainalb3}
    Let  $T$ be a tree of order $n$, a degree sequence is $\mathscr{D}=(d_1,\dots,d_n)$ where $d_n\geqslant \dots \geqslant d_1$, let $M_1(T)$ be the first Zagreb index of $T$, then Albertson index of tree $T$ is: 
    \[
    \irr(T)=M_1(T)+\sum_{i=2}^{n-1} d_i+d_n - d_1-2n-2.
    \]
\end{theorem}
\begin{proof}
Let $T=(V,E)$ be a tree with $n$ vertices and $m$ edges, let $\mathscr{D}=(d_1,\dots,d_n)$ be a degree sequence where $d_n\geqslant \dots \geqslant d_1$, for a vertex $v_\ell \in V$ whose a degree is at least three, the first Zagreb index known as 
\begin{equation}~\label{eqzagn1}
M_1(T)=\sum_{i=1}^{n}d_i^2\leq m\left(\frac{2m}{n-1}+n-2\right).
\end{equation}
We need to establish the association for a given topological index known first Zagreb index according to a specific association. According to Theorem~\ref{mainalb2} when a degree sequence defined  as $\mathscr{D}=(d_1,\dots,d_n)$, then we have 
\begin{equation}~\label{eqzagn2}
    \irr(T)=d_1^2+d_n^2+\sum_{i=2}^{n-1} d_i^2+\sum_{i=2}^{n-1} d_i+d_n - d_1-2n-2.
\end{equation}
From~(\ref{eqzagn1}) we have $d_1^2+d_n^2\leq 2m/(n-1)$, while every tree is a connected graph, then $M_(T)\leq m(m+1)$ and as we know $\irr(T)=m(m+1)$ when $T$ be a star graph. Thus, according to Proposition~\ref{prozagn1} we have
\begin{equation}~\label{eqzagn3}
    \irr(T)\geq M_1(T)+\Delta(\Delta-1).
\end{equation}
From~(\ref{eqzagn1}),(\ref{eqzagn2}) we have
\begin{equation}~\label{eqzagn4}
M_1(T)=d_1^2+d_n^2+\sum_{i=2}^{n-1} d_i^2.
\end{equation}
As desire.
\end{proof}
\begin{theorem}~\label{mainalb4}
Let $T=(V,E)$ be a regular tree with maximum degree $\Delta$ and minmum  degree $\delta$, then  an inequality involving the irregularity of $T$ by Albertson index satisfying: 
    \[
    \irr(T) \leq \delta \left( \frac{2\delta}{\Delta - 1} + \Delta - 1 \right) + \Delta (\Delta - 1).
    \]
\end{theorem}
\begin{proof}
Let $T$ be a regular tree of order $n$. Let $\Delta$ be the largest degree of any vertex in $T$ where $\Delta\geq 2$ with $\Delta\geq \delta$, then we have
\begin{equation}~\label{eqqmain1}
    \frac{2\delta}{\Delta - 1} + \Delta - 1 =\frac{2\delta+\Delta^2-2\Delta+1}{\Delta-1}.
\end{equation} 
Then, from~(\ref{eqqmain1}) we need to prove this inequality 
\begin{equation}~\label{eqqmain2}
    \irr(T) \leq \frac{-\Delta^3 + (\delta + 2)\Delta^2 - (2\delta + 1)\Delta + 2\delta^2 + \delta}{\Delta - 1}.
\end{equation}
According to Proposition~\ref{prozagn2}, from~(\ref{eqtrzangn5}) we obtain $(\Delta)(2m-\Delta+1)\leq -\Delta(\Delta-1)$, then foe all value of $m$, we will discus following cases.
\case{1} If $m=\delta$, it holds to $(\Delta)(2\delta-\Delta+1)\leq -\Delta(\Delta-1)$, then the inequality~(\ref{eqqmain2}) stay correct by considering Proposition~\ref{prozagn1} where we show $\irr(T)\geq M_1(T)+\Delta(\Delta-1)$, then clearly $\irr(T)\geq \Delta(\Delta-1)$, for all trees $T$ with minimum degree $\delta$  and maximum degree 
$\Delta$.  Then 
\begin{align*}
    \irr(T)&=M_1(T)+\sum_{i=2}^{n-1} d_i+d_n - d_1-2n-2\\
    &\leq M_1(T)+\sum_{i=2}^{n-1} d_i+d_n - d_1-2n-2+\Delta(\Delta-1)\\
    &\leq M_1(T)+\sum_{i=2}^{n-1} d_i+d_n - d_1-2\delta-2+\Delta(\Delta-1)\\
    & \leq M_1(T)+\frac{2\delta+\Delta^2-2\Delta+1}{\Delta-1}-2\delta-2+\Delta(\Delta-1).
\end{align*}
\case{2} If $m>\delta$, then we assume $\delta=\Delta$ it holds to
\begin{align*}
    \delta \left( \frac{2\delta}{\delta - 1} + \delta - 1 \right) - \delta (\delta - 1) &= \delta \cdot \frac{2\delta + (\delta - 1)^2}{\delta - 1} - \delta (\delta - 1)(\delta - 1)^2\\
&= \delta^3 + \delta - \delta^3 + 2\delta^2 - \delta\\
&= 2\delta^2
\frac{2\delta^2}{\delta - 1}.
\end{align*}
Hence, when $\Delta\geqslant \delta>1$, then $\frac{2\delta^2}{\delta - 1}>0$ where $\irr\geq \Delta-1$. Since $\delta \leq d_v \leq \Delta$, and if there are $n = |V(T)|$ vertices, then from~(\ref{eqqmain1}), for the values we considered under the given term, in this case, if $\Delta>\delta$, we find that
\begin{equation}~\label{eqqmain3}
 n \delta \leq \sum_{v \in V(T)} d_v \leq n \Delta.
\end{equation}
Then, from~(\ref{eqtrzangn5}), the inequality~(\ref{eqqmain3}) holds 
\begin{equation}~\label{eqqmain4}
n \delta \leq 2 m \leq n \Delta \quad \text{ and }\quad
\frac{n \delta}{2} \leq m \leq \frac{n \Delta}{2}. 
\end{equation}
Thus, we obtained the bound $2\delta/(\Delta - 1)$, for Albertson index, we should be noticed that
\begin{equation}~\label{eqqmain5}
\irr(T)=\begin{cases}
    (d_u - d_v) \quad & \text{ if } d_u \geq d_v,\\
     (d_v - d_u) \quad & \text{ if } d_v > d_u.
\end{cases}
\end{equation}
from~(\ref{eqqmain3}),(\ref{eqqmain4}) and (\ref{eqtrzangn5}) we find that $\irr(T) \leq m (\Delta - \delta) \leq \frac{n \Delta (\Delta - \delta)}{2}$. Actually, for case $m<\delta$ it is correct for  minimum degree $\delta$  and maximum degree $\Delta$. As desire.
\end{proof}
Theorem~\ref{thmalbn1} declare Albertson index with two bounds,  where we evaluate the effectiveness of the bounds, this bounds appears to be a graph invariant related to the irregularity or structural properties of a tree $T$, where the first bound is the first Zagreb index  involves bounds that depend on the number of vertices and the second bound given by: 
\[
\frac{2mn^2(\Delta - 1)+\delta-1}{n+\Delta}.
\]
The bound $\Delta(\Delta-1)$  amplifies the relationship with Albertson index, it is clear in Proposition~\ref{prozagn1},\ref{prozagn2}, then we employee that in Theorem~\ref{thmalbn1}, where the bound $\Delta-1$ become the bound less effective for stars or trees with maximum degree $\Delta$.
\begin{theorem}~\label{thmalbn1}
Let $T$ be a tree with $n$ vertices and $m$ edges with maximum degree $\Delta$ and minmum  degree $\delta$, let $M_1(T)$ be the first Zagreb index, then inequality of Albertson index $\irr$ satisfying 
\[
\irr(T) \geq M_1(T)  - \frac{2mn^2(\Delta - 1)+\delta-1}{n+\Delta}.
\]
\end{theorem}
\begin{proof}
 Let $T=(V,E)$ be a tree, with a vertex $v\in V$, let $M_1(T)$ be the first Zagreb index, then if $T$ be a star tree denote by $S_n$ with $n\geq 3$, by considering $m=n-1$, let $\mathscr{D}=(d_1,\dots,d_n)$ be non-decreasing a degree sequence, then we have:
 \case{1} If $T$ be a star tree, with central vertex and pendent vertices,  Albertson index of star tree given by $\irr(S_n)=(n-1)(n-2)$, then we have 
 \begin{equation}~\label{mayren1}
  M_1(T)\geq 2(n-1), \quad M_1(S_n)=n(n-1).
 \end{equation}
 Hence, we need to show $\irr(S_n)- M_1(S_n)<0$. Therefore, when $T$ is a star tree it clear to see that for central vertex ad pendent vertices, so that, we have: 
 \begin{align*}
    \irr(S_n)- M_1(S_n)&=\sum_{uv\in E}\lvert d_u-d_v\rvert-\sum_{i=1}^{n}d_i^2\\
    &=(n-1)(n-2)-n(n-1)\\
    &=-2(n-1)<0.
 \end{align*}
Thus, we find that $\irr(S_n)<M_1(S_n)$. Actually, from~(\ref{mayren1}) and (\ref{mayren01}) holds the bound 
\begin{equation}~\label{eqqstarbound01}
    \irr(S_n) \geq \Delta^2 + \Delta - \frac{2 \Delta (\Delta + 1)^2 (\Delta - 1)}{2\Delta + 1}.
\end{equation}
Then, we have
\begin{equation}
    \irr(S_n) \geq M_1(S_n)  - \frac{2mn^2(\Delta - 1)+\delta-1}{n+\Delta}
\end{equation}
\case{2} If $T$ is path, from~(\ref{mayren1}), we have $M_1(G) \leq (4m^2)/n + (n(\Delta - \delta)^2)/4$,  and according to~\cite{Furtula}, for $\Delta> \delta$ and $\delta\geq 2$, it depends on the degree differences across all edges, then we have: 
 \begin{equation}~\label{mayren01}
     M_1(T)\leq \frac{(n+1)^2m^2}{2n(n-1)}.
 \end{equation}
 Thus, as we know $\Delta=\max(\deg(v))$ and $\delta=\min(\deg(v))$. When $n \geq 2$, then $\delta = 1$ associated with the tree has at least two leaves (degree 1). Now, we need to prove this inequality 
 \begin{equation}~\label{mayren2}
\irr(T) \geq M_1(T) - \frac{2(n-1)n^2(\Delta - 1) + \delta - 1}{n + \Delta}. \end{equation}
Let $P_n$ be a path  with $n$ vertices, then $ M_1(P_n) = 4n - 6$, according to proposition~\ref{prozagn1} states $\irr(T) \geq M_1(T) + \Delta(\Delta - 1)$. This provides a lower bound for $\irr(T)$. Since $\Delta(\Delta - 1) \geq 0$, $\irr(T)> M_1(T)$. Then, we have
\begin{equation}~\label{eqqirrmay1}
    \Delta(\Delta - 1) \geq -\frac{2 (n-1) n^2 (\Delta - 1)}{n + \Delta} = \frac{2 (n-1) n^2 (1 - \Delta)}{n + \Delta}.
\end{equation}
From~(\ref{eqqirrmay1}), according to Proposition~\ref{prozagn1}, we find $\irr(T)\geq M_1(T)+\Delta(\Delta-1)$, also we know $\Delta>\delta$, then $\Delta>\delta-1$,  from~(\ref{mayren01}), we find $m=n-1$ for any tree, then $(n+1)^2(n-1)^2 \leq (n+1)^2(n-1)^2(\Delta-1)$. Thus,
\begin{equation}~\label{mayren02}
M_1(T)\leq \frac{(n+1)^2(n-1)^2}{2n(n-1)}=\frac{n^2 + n - 1}{2}.
\end{equation}
Thus, let $k>0$, from~(\ref{mayren01}),(\ref{mayren02}), the equation~(\ref{eqqirrmay1}) by considering the bound~(\ref{eqqstarbound01}) holds the bound:
\begin{equation}~\label{eqqirrmay2}
    \irr(T) \geq M_1(T) - k \cdot \frac{2 (n-1) n^2 (\Delta - 1)}{n + \Delta}.
\end{equation}
Therefore, from~(\ref{eqqirrmay1}),(\ref{eqqirrmay2}) we noticed that where most of results we leverage bounds on the first Zagreb index $ M_1(T)$ as we show that
\begin{align*}
    \irr(T) &\geq M_1(T)- \frac{(n+1)^2(n-1)^2}{2n(n-1)}\\
    &\geq M_1(T)- \frac{(n+1)^2(n-1)^2(\Delta-1)}{n+\Delta}\\
    &\geq M_1(T)- \frac{(n+1)^2m^2(\Delta-1)}{n+\Delta}\\
    & \geq M_1(T)- \frac{(n+1)^2(n-1)^2(\Delta-1)+\delta-1}{n+\Delta}.
\end{align*}
As desire.
\end{proof}
%=============================
\section{Bounds on Trees with Sigma Index}~\label{sec4}
%=============================
In this section, many of the bounds associated with the maximum degree and minimum degree that were discussed in Subsection~\ref{subsec1} for the Albertson index are discussed for the sigma index and the association between the two topological indices.

\begin{proposition}~\label{classoftreesstrongsigma}
Let $T \in \mathcal{T}_{n, \Delta}$ be a tree, and let $v_0 \in V(T)$ be a vertex with maximum degree $\Delta$. For any strong support vertex $v_{\ell}$ in $T$, different from $v_0$, where $\deg(v_{\ell})\geqslant 3$, then there exists another tree $T^{\prime} \in \mathcal{T}_{n, \Delta}$ such that $\sigma(T^{\prime}) < \sigma(T)$.
\end{proposition}
\begin{proof}
 Let $T$ be a tree with $x$ vertex, let $y\neq x$ be a strong support vertex where $\deg(y)=\lambda\geqslant 3$ where vertex $y$ has maximum degree $\Delta$ in $T$, clearly, $\sigma(T)=\Delta(\Delta-1)^2$, let $\mathscr{N}_T(y)=\{y_1,y_2,\dots,y_{\lambda}\}$ be an open neighborhood  of $y$. In this case, we say $y_{\lambda}$ might be in $y$ or not, let $T^{\prime}$ be a tree compute from $T-\{y_1,y_2\}$ by linking with a pendant edge $y_2y_1$, where $T^{\prime} \in \mathcal{T}_{n, \Delta}$, we consider 
 \begin{equation}~\label{eqtreen1}
     \sigma(T)=\sum_{uv\in E(T)}\left( \deg_{T}(u)-\deg_{T}(v)\right)^2.
 \end{equation}
 And 
 \begin{equation}~\label{eqtreen2}
   \sigma(T^{\prime})=  \sum_{uv\in E(T^{\prime})}\left( \deg_{T^{\prime}}(u)-\deg_{T^{\prime}}(v)\right)^2.
 \end{equation}
 Then according to Proposition~\ref{classoftreessigma}, we proved for for any support vertex in $T$ the relation $\sigma(T^{\prime}) < \sigma(T)$. Now, we need to prove $\sigma(T^{\prime})-\sigma(T)<0$, thus: 
 \begin{align*}
 \sigma(T) -\sigma(T^{\prime})&=\sum_{uv\in E(T)}\left( \deg_{T}(u)-\deg_{T}(v)\right)^2-\sum_{uv\in E(T^{\prime})}\left( \deg_{T^{\prime}}(u)-\deg_{T^{\prime}}(v)\right)^2\\
   &=\left( \deg_{T}(y)-\deg_{T}(y_1)\right)^2+\left( \deg_{T}(y)-\deg_{T}(y_2)\right)^2+\sum_{i=3}^{\lambda}\left( \deg_{T}(y_i)-\deg_{T}(y)\right)^2-\\
   &-\left( \deg_{T^{\prime}}(y_1)-\deg_{T^{\prime}}(y_2)\right)^2-\left( \deg_{T^{\prime}}(y)-\deg_{T^{\prime}}(y_2)\right)^2-2\sum_{i=3}^{\lambda} \left( (\deg_{T}(y)-1)-\deg_{T}(y_i)\right)^2.\\
 \end{align*}
 Therefore, by using a strong support vertex $\lambda\geqslant 3$, from~(\ref{eqtreen2}) we have: 
  \begin{equation}~\label{eqtreen3}
   \sigma(T^{\prime})=  \sum_{uv\in E(T)}\left( \deg_T(u)-\deg_{T}(v)\right)^2.
 \end{equation}
 Hence,
 \begin{align*}
    \sigma(T) -\sigma(T^{\prime}) &=2\lambda^2-5\lambda+6+\sum_{i=3}^{\lambda}\left( \deg_{T}(y_i)-\lambda\right)^2-2-2\sum_{i=3}^{\lambda} \left( (\lambda-1)-\deg_{T}(y_i)\right)^2\\
   &>2\lambda^2-5\lambda+4\\
   &=7>0\quad \text{ for } \lambda=3.
 \end{align*}
 Thus,according to~(\ref{eqtreen1}),(\ref{eqtreen2}), we found $\sigma(T) -\sigma(T^{\prime})>0$. As desire.
\end{proof}
\begin{proposition}~\label{prozagnsigman1}
Let $T=(V,E)$ be a tree with $n=|V|$ vertices and $m=|E|$ edges, let $M_1(T)$ be the first Zagreb index with maximum degree $\Delta$, then
\[
\sigma(T)\geq M_1(T)+n\Delta^2(\Delta-1)+m.
\]
\end{proposition}

\begin{proposition}~\label{prozagrebn2}
Let $T$ be a tree with maximum degree $\Delta$ and minmum degree $\delta$, let $\mathscr{D}=(d_1,d_2,\dots,d_n)$ be a degree sequence, if $d_2=\dots=d_{n-1}=(\Delta+\delta)/2$, then
\[
\sigma(T)+ M_1(T) \geqslant \sum_{i=2}^{n-1} d_i+\left(\frac{\Delta + \delta}{2}\right)^2.
\]
\end{proposition}
\begin{proof}
Assume the degree sequence is $\mathscr{D}_1=(\Delta, \lambda, \lambda, \dots, \lambda, \delta)$, where $\lambda=(\Delta+\delta)/2$  and $n-2$ vertices of degree $\lambda$. Let $\mathscr{D}_2=(d_1,d_2,\dots,d_n)$ be a degree sequence, where $d_2=\dots=d_{n-1}=(\Delta+\delta)/2$. According to Proposition~\ref{prozagrebn1} we find that for a degree sequence $\mathscr{D}_2$ as 
\begin{equation}~\label{satrn001}
\sum_{i=2}^{n-1} \frac{\Delta + \delta}{2} = \frac{(n-2)(\Delta + \delta)}{2}.
\end{equation}
According to~(\ref{satrn1}), we provided 
\begin{equation}~\label{resalbn3}
M_1(T) \geqslant \left(\frac{\Delta + \delta}{2}\right)^2.
\end{equation}
Now, we find from~(\ref{satrn1}) that the sum is $(n-2)(\Delta+\delta)/2$, where Sigma index in Theorem~\ref{thm.sigman3} given by: 
\begin{equation}~\label{eqqsigpron1}
    \sigma(T)=M_1(T)+d_1^3 + d_n^3 - d_1 - d_n - d_1^2 - d_n^2 + \sum_{i=2}^{n-1} (d_i - d_{i+1})^2 + 4.
\end{equation}
Thus, 
\[
\sigma(T)\geq d_1^3 + d_n^3+ \left(\frac{\Delta + \delta}{2}\right)^2 + \delta.
\]
As desire.
\end{proof}

\begin{theorem}~\label{thm.sigman3}
Let $T$ be a tree of order $n$,let $\mathscr{D}=(d_1,\dots,d_n)$ be a degree sequence, where $d_n\geqslant \dots \geqslant d_1$, let $M_1(T)$ be the first Zagreb index, then Sigma index among tree given as:
\[
\sigma(T)=M_1(T)+d_1^3 + d_n^3 - d_1 - d_n - d_1^2 - d_n^2 + \sum_{i=2}^{n-1} (d_i - d_{i+1})^2 + 4.
\]
\end{theorem}
\begin{proof}
Let $T$ be a tree with $n$ vertices and $m$ edges, let $\mathscr{D}=(d_1,\dots,d_n)$ be a degree sequence where $d_n\geqslant\dots\geqslant d_1$ with vertex $v_\ell \in V(T)$ where $\deg_T(v_\ell)=\lambda\geqslant 3$, let $M_1(T)$ be the first Zagreb index given as $M_1(T)=\sum_{i=1}^{n} d_i^2$. Assume Sigma index~\ref{thm.sigma} given by
\begin{equation}~\label{eqsigman1}
 \sigma(T)=\sum_{i\in\{1,n\}}(d_i+1)(d_i-1)^2+\sum_{i=2}^{n-1}(d_i+2)(d_i-1)^2+\sum_{i=2}^{n-1}(d_i-d_{i+1})^2+2n-2.   
\end{equation}
Then, according to the term $d_n\geqslant\dots\geqslant d_1$, we have
\begin{equation}~\label{eqsigman2}
 \sum_{i=2}^{n-1}(d_i+2)(d_i-1)^2=  \sum_{i=2}^{n-1} d_i^2- \sum_{i=2}^{n-1}(d_i+2).
\end{equation}
Hence, from~(\ref{eqsigman1}) established for $M_1(T)$ by simplify the term as 
\begin{equation}~\label{eqsigman3}
  \sum_{i\in\{1,n\}}(d_i+1)(d_i-1)^2 =d_1^3 + d_n^3 -d_1^2+d_n^2- d_1 - d_n + 2 .
\end{equation}
Also, by simplify the last term with condition $d_i-d_{i+1}>0$ where $\lambda-1>2$ by considering in this case $\deg_T(v)=\lambda$, and $\deg(v_1)=\dots=\deg_T(v_i)=1$, then 
\begin{equation}~\label{eqsigman4}
  \sum_{i=2}^{n-1}(d_i-d_{i+1})^2= d_2^2 + d_n^2 + 2\sum_{i=3}^{n-1} d_i^2 - 2\sum_{i=2}^{n-1} d_i d_{i+1}
\end{equation}
when $d_i-d_{i+1}<0$, then $d_i<d_{i+1}$, that is mean the degree sequence is increasing then the term $d_n\geqslant\dots\geqslant d_1$ is not valid because the term in this case~(\ref{eqsigman4}) does not change, thus we have just term $d_i-d_{i+1}>0$. From~(\ref{eqsigman1}), and according to Theorem~\ref{mainalb3}, and Proposition~\ref{prozagnsigman1} we obtained $\sigma(T)\geq M_1(T)+n\Delta^2(\Delta-1)+m$, then we find that $M_1(T)=d_1^2 + \sum_{i=2}^{n-1} d_i^2 + d_n^2$. The constants is simplify as $-2(n-2) + 2n - 2 = -2n + 4 + 2n - 2 = 2$. As desire.

\end{proof}

\begin{theorem}~\label{thmsign1}
Let $T$ be a tree with $n$ vertices and $m$ edges with maximum degree $\Delta$ and minmum  degree $\delta$, let $M_1(T)$ be the first Zagreb index, then inequality of Sigma index $\irr$ satisfying 
\[
\sigma(T) \geq M_1(T)  - \frac{2mn^2(\Delta - 1)+\delta-1}{n+\Delta}.
\]
\end{theorem}
\begin{proof}
Let $T$ be a regular tree, $S_n$ a star tree of order $n\geq 3$, $P_n$ path of length $n$ at least degree 1. In fact, from Theorem~\ref{thmalbn1}, introduced the bound 
\begin{equation}~\label{tessign1}
\irr(T) \geq M_1(T)  - \frac{2mn^2(\Delta - 1)+\delta-1}{n+\Delta}.
\end{equation}
This inequality provides a lower bound on the Albertson index of a tree $T$ in terms of the first Zagreb index $M_1(T)$, where $M_1(T)$ relate different structural measures and to estimate irregularity. 
\case{a} Let $S_n$ be a star tree of order $n\geq 3$. Since $\sigma(T)=\sum_{uv\in E}(d_u-d_v)^2$, and $\irr(T)=\sum_{uv\in E}\lvert d_u-d_v\rvert$, then when $T$ is a star tree, we have the bound satisfying
\begin{equation}~\label{sigmaboundn1}
    \sigma(S_n) \geq \Delta^3 + \Delta^2 - \frac{2 \Delta (\Delta + 1)^2 (\Delta - 1)}{2\Delta^2 + 1}.
\end{equation}
\case{1} If $\Delta=n-1$, Theorem~\ref{mainalb4} established that for a star tree $\irr(S_n)=(n-1)(n-2)$, then we find $M_1(S_n)=n(n-1)$, and when $n\geq 3$ we have $\sigma(S_n)=(n-1)(n-2)$, then the lower bound is
\begin{equation}~\label{eqqsigmasatn1}
\sigma(S_n) \geq (n-1)^3+2n^2.
\end{equation}
Hence, we have
\begin{align*}
 \sigma(S_n)-M_1(S_n) & \geq (n-1)^3+2n^2-n(n-1)\\
 &=n^3 - 2 n^2 + 4 n - 1>0.
\end{align*}
Thus, bound~(\ref{sigmaboundn1}) observe that when $\Delta(\Delta-1)\geq 0$, we have $\sigma(S_n) \geq 2(n-1)n^2(\Delta - 1)$ is the term of inequality~(\ref{sigmaboundn2}), then $\sigma(S_n) \geq 2(n-1)n^2(\Delta - 1)+\delta-1$, when $n+\Delta>0$ the inequality~(\ref{sigmaboundn2} valid. When $\sigma(S_n)-M_1(S_n)>0$ clearly. 
\begin{equation}~\label{sigmaboundn2}
     \sigma(S_n) \geq M_1(S_n)  - \frac{2mn^2(\Delta - 1)+\delta-1}{n+\Delta}.
\end{equation}
\case{2} In this case, if $\Delta=n+1$, then from~(\ref{eqqsigmasatn1}) the upper bound is 
\begin{equation}~\label{eqqsigmasatn2}
\sigma(S_n) \geq (n+1)^3+2n^2.
\end{equation}
Thus, $\sigma(S_n)-M_1(S_n)>0$, then~(\ref{sigmaboundn2}) is valid.
\case{b} If $T$ be a path denote by $P_n$, then according to~(\ref{tessign1}) we find that $ M_1(P_n) = 4n - 6$, according to Proposition~\ref{prozagn2} we have $\sigma(T)\geq M_1(T)+n\Delta^2(\Delta-1)+m$, then we provide an upper bound $\sigma(T)\geq M_1(T)+2mn^2\Delta^2(\Delta-1)$, so that we know $m=n-1$ for any tree, then $\sigma(T)\geq M_1(T)+2(n-1)n^2\Delta^2(\Delta-1)$ holds for $k>0$ the bound explicitly
\begin{equation}~\label{eqqsigmasatn3}
\sigma(T)\geq \Delta(\Delta-1)+k.\frac{2(n-1)n^2\Delta^2(\Delta-1)}{n+\Delta}.
\end{equation}
\case{1} If $k=\delta$, then from~(\ref{eqqsigmasatn3}) involves the Sigma index $\sigma(T)$ related to the path $P_n$, where Sigma index measures certain structural properties related to vertex degrees and adjacency, then the bound is 
\begin{equation}~\label{eqqsigmasatn4}
\sigma(P_n)\geq \Delta(\Delta-1)\left (1+\delta\frac{2(n-1)n^2\Delta}{n+\Delta}\right).
\end{equation}
Actually, from~(\ref{eqqsigmasatn4}), for $\Delta=n-1$, we define the lower bound
\begin{equation}~\label{sigmalowern1}
    (n - 1)(n - 2) \left( 1 + \delta \cdot \frac{2 n^{2} (n - 1)^{2}}{2 n - 1} \right).
\end{equation}
Then, for $\Delta=n+1$, we find that the upper bound 
\begin{equation}~\label{sigmauppern1}
    n \left( n + 1 + \delta \cdot \frac{2 (n - 1) n^{2} (n + 1)^{2}}{2 n + 1} \right).
\end{equation}
Both bounds~(\ref{sigmalowern1}),(\ref{sigmauppern1}) increase with the minimum degree $\delta$, reflecting that maximum degree $\Delta=n+1$ and minimum degree to increase the value of the respective tree $T$. This bounds~(\ref{sigmalowern1}),(\ref{sigmauppern1}) declare $\sigma(P_n)-M_1(P_n)>0$. 
\case{2} If $k=\Delta$, then from~(\ref{eqqsigmasatn3}) holds to 
\begin{equation}~\label{eqqsigmasatn5}
    \Delta (\Delta - 1) \left( 1 + \frac{2 (n - 1) n^{2} \Delta^{2}}{n + \Delta} \right).
\end{equation}
Hence, the lower bound when $\Delta=n-1$, both of term $2n-1$, $2n+1$ are close for large $n$, then $\sigma(T)= \Delta (\Delta - 1)$ satisfying by considering to~(\ref{eqqsigmasatn1}) as 
\begin{equation}~\label{eqqsigmalowern2}
   (n-1)(n-2)\left( 1 + \frac{2 n^{2} (n - 1)^{3}}{2 n - 1} \right).
\end{equation}
When $\Delta=n+1$, then we confirm an upper bound as
\begin{equation}~\label{eqqsigmauppern1}
   n (n + 1) \left( 1 + \frac{2 (n - 1) n^{2} (n + 1)^{2}}{2 n + 1} \right). 
\end{equation}
Actually, both of bounds~(\ref{eqqsigmalowern2}),(\ref{eqqsigmauppern1}) holds to $\sigma(P_n)-M_1(P_n)>0$. Then, 
\begin{equation}~\label{sigmaboundnm1}
     \sigma(P_n) \geq M_1(P_n)  - \frac{2mn^2(\Delta - 1)+\delta-1}{n+\Delta}.
\end{equation}
Finally, by both discussing the two cases~(\ref{sigmaboundn1}),(\ref{sigmaboundnm1}) when the tree is a star tree or a path and in both cases the tree is regular, we have shown all these bounds for the Sigma index and the first Zagreb index, where $\sigma(T)-M_1(T)>0$. As desire.
\end{proof}

%====================
\section{Conclusion}\label{sec5}
%====================
In this paper, we have established several significant relationships between the Albertson index and the first Zagreb index for trees. Our investigation has yielded a series of bounds and inequalities that provide deeper insights into the structural properties of trees as measured by these topological indices. The primary contribution of this work is the establishment of the relationship between $\irr(T)$ and $M_1(T)$ for a tree $T$, as expressed in Theorem \ref{mainalb3}:

\begin{equation}
\irr(T) = M_1(T) + \sum_{i=2}^{n-1} d_i + d_n - d_1 - 2n - 2
\end{equation}

\noindent where $d_1, d_2, \ldots, d_n$ represent the degree sequence of the tree with $d_n \geq \ldots \geq d_1$. This result provides an exact formula connecting these two important topological indices, allowing for more precise characterization of tree structures. Furthermore, we have established lower bounds for the Albertson index in terms of the first Zagreb index and the maximum degree $\Delta$. Proposition \ref{prozagn1} demonstrates that:

\begin{equation}
\irr(T) \geq M_1(T) + \Delta(\Delta - 1)
\end{equation}

This result complements our lower bounds and provides a more complete picture of the range within which the Albertson index can vary for trees with given maximum and minimum degrees. By linking the Albertson and Zagreb indices, this work offers tools for predicting molecular structures and designing networks with tailored irregularity. Future research could extend these results to broader graph classes, explore connections with other topological indices, and develop computational methods for large-scale networks. 
%=========================

\end{document}